\documentclass[12pt,a4paper]{amsart}
\usepackage{mathrsfs}
\usepackage{amsfonts}

\usepackage[active]{srcltx}
\usepackage{enumerate}
\usepackage{color}
\def\bm{\mathbf{m}}
\def\bs{\mathbf{s}}
\def\ba{\mathbf{a}}
\def\bb{\mathbf{b}}
\def\br{\mathbf{r}}

\def\bo{\mathbf{0}}

\usepackage{amsmath,amssymb,xspace,amsthm}
\newtheorem{theorem}{Theorem}
\newtheorem{remark}[theorem]{Remark}
\newtheorem{lemma}[theorem]{Lemma}
\newtheorem{proposition}[theorem]{Proposition}
\newtheorem{corollary}[theorem]{Corollary}
\numberwithin{equation}{section}

\usepackage{color}
 \usepackage[colorlinks,final,backref=page,hyperindex]{hyperref} 

\newcommand{\C}{\mathbb C}

\newcommand{\N}{\mathbb{N}}
\newcommand{\Z}{\mathbb{Z}}

\input xy
\xyoption{arrow} \xyoption{matrix}

\def\mg{\mathfrak{g}}
\def\n{\mathfrak{n}}

\def\b{\mathfrak{b}}
\def\mh{\mathfrak{h}}

\def\mm{\mathfrak{m}}
\def\m{\mathbf{m}}
\def\r{\mathbf{r}}
\def\c{\mathbf{c}}
\def\sl{\mathfrak{sl}}
\def\gl{\mathfrak{gl}}
\def\sp{\mathfrak{sp}}

\newcommand{\cd}{\mathcal{D}}
\newcommand{\e}{\epsilon}
\def\p{\partial}

\title[Whittaker representations]{\bf A  Whittaker
category for
the Symplectic Lie algebra}
\author{Yang Li, Jun Zhao, Yuanyuan Zhang,  Genqiang Liu}

\date{\today}

\begin{document}

\begin{abstract}For any $n\in \mathbb{Z}_{\geq 2}$, let $\mathfrak{m}_n$ be the subalgebra of $\mathfrak{sp}_{2n}$
spanned by all long negative  root vectors $X_{-2\epsilon_i}$, $i=1,\dots,n$. An $\mathfrak{sp}_{2n}$-module $M$ is called a  Whittaker module with respect to the Whittaker pair $(\mathfrak{sp}_{2n},\mathfrak{m}_n)$ if the action of $\mathfrak{m}_n$ on $M$ is locally finite, according to a definition  of Batra and Mazorchuk. This kind of modules are more general than the classical  Whittaker modules
defined by Kostant.
In this paper, we show  that each non-singular block   $\mathcal{WH}_{\mathbf{a}}^{\mu}$ with finite dimensional Whittaker vector subspaces  is equivalent to a module category  $\mathcal{W}^{\mathbf{a}}$ of the even Weyl algebra
$\mathcal{D}_n^{ev}$ which is semi-simple. As a corollary, any simple module in the block $\mathcal{WH}_{\mathbf{i}}^{-\frac{1}{2}\omega_n}$ for the fundamental weight $\omega_n$ is equivalent to the Nilsson's module $N_{\mathbf{i}}$ up to an automorphism of $\mathfrak{sp}_{2n}$.
 We also characterize all possible algebra homomorphisms from $U(\mathfrak{sp}_{2n})$ to the Weyl algebra $\mathcal{D}_n$ under a natural condition.
\end{abstract}
\vspace{5mm}
\maketitle

\vspace{2mm}

\noindent{{\bf Math. Subj. Class.} 2020: 17B05, 17B10, 17B30, 17B35}

\section{introduction}

Among the representation theory of Lie algebras, Whittaker modules are
interesting non-weight modules which play an important role in the classification of
irreducible modules for several Lie algebras.
Whittaker modules were first introduced by Arnal and Pinzcon for $\mathfrak{sl}_2(\C)$, see \cite{AP}. The classification of the irreducible modules for $\mathfrak{sl}_2(\C)$ in \cite{B} illustrates the importance of Whittaker modules.
 It was shown  that irreducible $\mathfrak{sl}_2(\C)$-modules can be divided into three families: weight modules, Whittaker modules, and modules obtained from irreducible elements in a noncommutative domian. Kostant studied  Whittaker modules for any  complex semisimple Lie algebra $\mg$ in \cite{K}. Whittaker modules defined by Kostant are closely associated with the triangular decomposition $\mathfrak{n}_- \oplus \mathfrak{h} \oplus \mathfrak{n}_+$ of  $\mathfrak{g}$. Every Whittaker module depends on a Lie
algebra homomorphism $\psi : \mathfrak{n}_+ \rightarrow\C$.
 The map $\psi$ is called non-singular if $\psi(x_\alpha)\neq 0$ for any simple root vector $x_\alpha$. Kostant gave a classification of all simple non-singular Whittaker modules. Some results on complex semisimple Lie algebras have been generalized to other algebras with  triangular decompositions. For example, for Whittaker modules over algebras related to the Virasoro algebra, one can see \cite{OW1,OW2,GLZ, LPX,LWZ}.  Whittaker modules over quantum groups $U_h(\mathfrak{g})$, $U_q(\mathfrak{sl}_2)$ and $U_q(\mathfrak{sl}_3) $ were studied in \cite{S,OM, XGZ}, respectively. Whittaker modules have also been studied for generalized Weyl
algebras by Benkart and Ondrus, see  \cite{BO}.  Whittaker modules for non-twisted affine
Lie algebras and several similar  algebras were studied  in \cite{ALZ, C,CF, CJ,GZ}.  In \cite{BM},  Batra and Mazorchuk  have constructed a more
general framework to describe the Whittaker modules. They considered Whittaker pairs $(\mathfrak{g}, \mathfrak{n})$ of Lie algebras, where $\mathfrak{n}$ is a quasi-nilpotent Lie subalgebra of $\mathfrak{g}$ such that the adjoint action of $\mathfrak{n}$ on the quotient  $\mathfrak{g}/\mathfrak{n}$ is locally nilpotent, and
studied the category $\mathcal{WH}$ of $\mathfrak{g}$-modules such that $\mathfrak{n}$ acts locally finitely. Under this general Whittaker set-up in \cite{BM}, they also determined a block decomposition of the category $\mathcal{WH}$ according to the action of $\mathfrak{n}$. The characterizations of each block for most Lie algebras are still open.

Differential operators are important tools for studying representations for Lie algebras. To construct explicit representations of a Lie algebra $\mg$ by differential operators, it is actually to find  algebra homomorphisms from $U(\mg)$ to the Weyl algebra $\cd_n$.  Let $\mm_n$ be the subalgebra of $\sp_{2n}$ spanned by root vectors $X_{-2\epsilon_i}$, $i\in\{1,\dots,n\}$. Then $(\sp_{2n},\mm_n)$ is a Whittaker pair. An $\sp_{2n}$-module $M$ is called a  Whittaker module if the action of each element of $\mm_n$ on $M$ is locally finite, see \cite{BM}. Similar as Kostant's definition, a Lie algebra homomorphism $\phi: \mm_n\rightarrow \C$ is called non-singular if $\phi(X_{-2\epsilon_i})\neq 0$, for any $i\in\{1,\dots,n\}$. We will characterize non-singular  Whittaker modules in this paper using differential operators.

 The paper is organized as follows. In Section 2, we recall some basic definitions and important facts including the weighting functors and Nilsson's modules for $\sp_{2n}$. In Section 3, we characterize the Whittaker category $\mathcal{WH}_{\mathbf{a}}$  when $\mathbf{a}\in (\C^*)^n$, where $\mathcal{WH}_{\mathbf{a}}$ consists of $\sp_{2n}$-modules $M$ such that $X_{-2\epsilon_i}-a_i^2 $ acts locally nilpotently  for any $i=1,\dots,n$, and $\mathrm{wh}_{\mathbf{a}}(M)=\{v\in M \mid  X_{-2\epsilon_i}v=a_i^2v, \ i=1,\dots,n \}$ is finite dimensional. Let $\mathcal{WH}_{\mathbf{a}}^{\mu}$ be the full subcategory of $\mathcal{WH}_{\mathbf{a}}$ consisting of all $U(\sp_{2n})$-modules $M$ with the central character $\chi_{\mu}$ given by the highest weight $\mu$.  We  show that when $\mathcal{WH}_{\mathbf{a}}^{\mu}$ is non-empty, there is an equivalence between $\mathcal{WH}_{\mathbf{a}}^{\mu}$ and the category $\mathcal{W}^{\mathbf{a}}$ of finitely generated
$\cd_n^{ev}$-modules such that $\p_i^2-a_i^2$ acts locally nilpotently for any $i$, where $\cd_n^{ev}$ is the subalgebra of the Weyl algebra $\cd_n$ generated by differential operators of even degree. In Section 4, we  give a differential operators realization of $\sp_{2n}$ from any $f\in A_{n}$, see Lemma \ref{hom1}. Thus we have constructed many simple modules $P_n^f$ over $\sp_{2n}$. Furthermore, we show that these operators realizations exhaust all algebra  homomorphisms  from $U(\sp_{2n})$ to $\cd_n$ which map each root vector $X_{\e_i+\e_j}$  to $t_it_j$.

In this paper, we denote by $\Z$, $\N$, $\Z_+$, $\C$ and $\C^*$ the sets of integers, positive integers, nonnegative
integers, complex numbers, and nonzero complex numbers, respectively. All vector spaces and algebras are over $\C$. For a Lie algebra
$\mathfrak{g}$ we denote by $U(\mathfrak{g})$ its universal enveloping algebra, $Z(\mathfrak{g})$ the center of $U(\mathfrak{g})$. We write $\otimes$ for $\otimes_{\mathbb{C}}$.

\section{Preliminaries}

In this section, we collect some preliminary definitions and related results that will be used throughout the paper.
In particular, we introduce the notion of Whittaker modules in the sense of \cite{BM}   that are of main interest in this paper.

\subsection{The symplectic algebra $\mathfrak{sp}_{2n}$}
Throughout the whole text, we fix an integer  $n$ bigger than $1$.
Recall that    $\mathfrak{sp}_{2n}$ is the Lie subalgebra of $\mathfrak{gl}_{2n}$ consisting of all $2n\times 2n$-matrices $X$ satisfying $SX=-X^{T}S$ where
\begin{displaymath}
 S=\left( \begin{array}{cc}
  0 & I_{n}\\
  -I_{n} & 0\\
 \end{array} \right).
\end{displaymath}
So $\mathfrak{sp}_{2n}$ consists of all $2n\times 2n$-matrices of the following form
\begin{displaymath}
 \left( \begin{array}{cc}
  A & B\\
  C & -A^T\\
 \end{array} \right)
\end{displaymath}
such that $B= B^{T}$,  $C= C^{T}$, where $A, B, C\in \gl_n$. Let $e_{ij}$ denote the matrix unit whose $(i,j)$-entry is $1$ and $0$ elsewhere.
Then $$\mathfrak{h}_n=\mathrm{span}
\{h_{i}:=e_{i,i} - e_{n+i,n+i} \mid 1 \leq i \leq n\}$$  is a Cartan subalgebra (a maximal
abelian subalgebra whose  adjoint action  on $\mathfrak{sp}_{2n}$ is diagonalizable)
of $\mathfrak{sp}_{2n}$.
Let $\Lambda^+$ be the set of dominant integral weight, $\{\epsilon_{i}\mid 1 \leq i \leq n\}\subset\mathfrak{h}_n^{*}$ be such that $\epsilon_{i}(h_{k})=\delta_{i,k}$.
The root system of $\mathfrak{sp}_{2n}$ is precisely \[\Delta= \{\pm \epsilon_{i} \pm \epsilon_{j} | 1 \leq i,j \leq n\} \setminus \{0\}.\]

The positive root system is
$$\Delta_+=\{\epsilon_{i}-\epsilon_{j}, \epsilon_{k}+\epsilon_{l} \mid 1\leq i<j\leq n, 1\leq k, l\leq n\}.$$

We list root vectors in $\mathfrak{sp}_{2n}$ as follows:
\begin{displaymath}
\begin{array}{r c l | c}
\multicolumn{3}{c}{\mathrm{{Root\ vector}}} &\mathrm{{Root} }\\
\hline
X_{\epsilon_{i}+\epsilon_{j}}  &:=& e_{i,n+j} + e_{j,n+i}      & \epsilon_{i}+\epsilon_{j}\\
X_{-\epsilon_{i}-\epsilon_{j}} &:=& e_{n+i,j} +e_{n+j,i}      & -\epsilon_{i}-\epsilon_{j}\\
X_{\epsilon_{i}-\epsilon_{j}}  &:=& e_{i,j} - e_{n+j,n+i}      & \epsilon_{i}-\epsilon_{j},\\
\end{array}
\end{displaymath}
where  $i,j\in \{1,\dots,n\}$, with $i\ne j$ when we encounter $\epsilon_{i}-\epsilon_{j}$.

Then we can obtain a basis of $\mathfrak{sp}_{2n}$ as follows: $$B:=\{X_{\alpha} | \alpha \in \Delta\} \cup \{h_{i}| 1 \leq i \leq n\}.$$


Set$$\n_{\pm}:=\bigoplus_{\alpha\in\Delta_\pm}\mg_\alpha.$$
Then the decomposition $$\mathfrak{sp}_{2n}=\n_-\oplus \mh_n\oplus \n_+$$
is a triangular decomposition of $\mathfrak{sp}_{2n}$, and the Lie subalgebra $\b:=\mh_n\oplus \n_+$ is a Borel subalgebra of $\mathfrak{sp}_{2n}$.

For the convenience of later calculations,
we list some  nontrivial Lie bracket of $\sp_{2n}$ as follows:
\begin{equation}\aligned\  [X_{\e_i-\e_j}, X_{\e_k-\e_l}]&=\delta_{jk}X_{\e_i-\e_l}-\delta_{li}X_{\e_k-\e_j} ,\\
[X_{\e_i+\e_j}, X_{-\e_k-\e_l}]&=\delta_{jk}X_{\e_i-\e_l}+\delta_{il}X_{\e_j-\e_k}
+\delta_{ik}X_{\e_j-\e_l}+\delta_{jl}X_{\e_i-\e_k} ,\\
 [X_{\e_i-\e_j}, X_{\e_k+\e_l}]&=\delta_{jk}X_{\e_i+\e_l}+\delta_{jl}X_{\e_i+\e_k},\\
 [X_{\e_i-\e_j}, X_{-\e_k-\e_l}]&=-\delta_{il}X_{-\e_k-\e_j}-\delta_{ki}X_{-\e_l-\e_j}.\\
\endaligned\end{equation}
In particular, $[X_{2\e_i}, X_{-2\e_k}]=4[e_{i,n+i},e_{n+k,k}]=\delta_{ik}4 h_i$, where $i,k \in\{1,\dots, n\}$.

\subsection{Weight modules}
An $\sp_{2n}$-module $V$  is  called a weight module if $\mathfrak{h}_n$ acts diagonally on $V$, i.e.,
$$V=\oplus_{\lambda\in \mathfrak{h}_n^*}V_\lambda,$$ where $V_\lambda=\{v\in V \mid h v=\lambda(h) v, \forall\ h\in \mathfrak{h}_n\}$.
For a weight module $V$, denote $$\mathrm{supp}(V)=\{\lambda\in\C^*|V_\lambda\neq0\}.$$

For a weight module $M$, a nonzero  vector $v \in M_\lambda$ is called a highest weight vector if $\n_+v = 0$. A module is called a highest weight module if it is generated by a highest weight
vector.
A weight module $M$ is  called a uniformly bounded module, if there is a
$k\in \N$ such that $\dim M_\lambda\leq k$ for any weight $\lambda\in\mathrm{supp}(M)$.
Let $\mathcal{B}$ be the category  consisting of uniformly
bounded weight modules.

\subsection{Whittaker modules}

Let  $\mm_n=\oplus_{ 1\leq i\leq n} \mathbb{C} X_{-2\epsilon_i}$ which is a commutative subalgebra of $\sp_{2n}$.
Since the adjoint action of $\mm_n$ on the quotient  $ \sp_{2n}/\mm_n$ is nilpotent, $(\sp_{2n},\mm_n)$ is a Whittaker pair in the sense of  \cite{BM}. An $\sp_{2n}$-module $M$ is called a Whittaker module if the action of $\mm_n$ on $M$ is locally finite. For an $\mathbf{a}=(a_1,\cdots,a_n)\in (\C)^n$, we can define a
Lie algebra homomorphism $\phi_{\mathbf{a}}: \mm_n \rightarrow \C$ such that $\phi_{\mathbf{a}}(X_{-2\e_i})=a_i^2$ for any $i \in \{1,\cdots,n\}$. A   Whittaker module $M$ is of type $\mathbf{a}$ if for any $v\in M$ there is a $k\in \N$ such that $(x-\phi_{\mathbf{a}}(x))^kv=0$ for all $x\in \mm_n$.
We also define the subspace
 $$\mathrm{wh}_{\mathbf{a}}(M)=\{v\in M \mid x v=\phi_{\mathbf{a}}(x)v, \ \forall\ x\in \mm_n \}$$ of $M$.
 An element in $\mathrm{wh}_{\mathbf{a}}(M)$ is called a Whittaker vector.

Such Whittaker modules are more complicated than the classical  Whittaker modules
defined by Kostant. For example, $\dim \text{wh}_{\mathbf{a}}(M)$ is not necessarily $1$ for a simple    Whittaker module $M$.  We consider Whittaker modules under some natural finite condition.  Let $\mathcal{WH}_{\mathbf{a}}$ be the category of   Whittaker $U(\sp_{2n})$-modules  $M$ of type $\mathbf{a}$ such that $\mathrm{wh}_{\mathbf{a}}(M)$ is finite dimensional.

\begin{remark} The condition that $\mathrm{wh}_{\mathbf{a}}(M)$ is finite dimensional
amounts to  the  condition that weight spaces are finite dimensional for a weight module, see the proof  of Lemma \ref{bounded}.
\end{remark}

\subsection{Central characters}

Let $\mathfrak{X}=\text{Hom}(Z(\sp_{2n}), \C)$ be the set of central characters of $\sp_{2n}$.  We have a map $ \chi : \mh_n^*\rightarrow \mathfrak{X} $ which maps $\mu \in \mh_n^*$ to the central character of the Verma module $M(\mu)$.
By the Harish-Chandra's Theorem, the map $\chi $ is surjective.  For each
${\mu}\in \mh_n^*$, denote by $\mathcal{WH}_{\mathbf{a}}^{\mu}$ the full subcategory of $\mathcal{WH}_\mathbf{a}$  of all $U(\sp_{2n})$-modules $M$ such that for any $v\in M$ there is a $k\in \N$ such that $(z-\chi_{\mu}(z))^kv=0$ for all $z\in Z(\sp_{2n})$.
Similarly, we have the full subcategory $\mathcal{B}^{\mu}$ of $\mathcal{B}$ for any $\mu\in \mh_n^*$. Moreover we have the block decompositions:
$$\mathcal{WH}_{\mathbf{a}}=\oplus_{\mu\in \mh_n^*} \mathcal{WH}_{\mathbf{a}}^{\mu}, \ \ \mathcal{B}=\oplus_{\mu\in \mh_n^*}\mathcal{B}^{\mu}.$$

\subsection{Weighting functor}
We recall the weighting functor introduced in \cite{N2}.
For a point $\gamma\in\C^n$,  let $I_\gamma$ be the maximal ideal of  $U(\mh_n)=\C[h_1,\dots,h_n]$  generated by $$h_1-\gamma_1,\dots, h_n-\gamma_n.$$
For an $\sp_{2n}$-module $M$ and  $\gamma \in\C^n$, set $M^{\gamma }:= M/I_{\gamma}M$.
Let  $$\mathfrak{W}(M):=\bigoplus_{\gamma\in\Z^n}M^{\gamma-\frac{1}{2}\omega_n},$$
where $\omega_n=(1,\dots,1)$.

For any $\lambda \in \mh_n^*$, we identify $\lambda$
with the vector $(\lambda(h_1),\cdots,\lambda(h_n))$ in $\C^n$. Nilsson defined a weight
 module structure on $\mathfrak{W}(M)$, see Proposition 8 in \cite{N2}.
\begin{proposition} The vector space $\mathfrak{W}(M)$  becomes a weight $\sp_{2n}$-module   under the following action:
\begin{equation}\label{3.3}
X_{\alpha} \cdot(v+I_{\gamma}M):= X_{\alpha} v+I_{\gamma+\alpha}M, v\in M, \alpha\in \Delta.
\end{equation}
\end{proposition}

We see that $h_i\cdot(v+I_{\gamma}M)=\gamma_i (v+I_{\gamma}M)$ for any $i$. So  $\mathfrak{W}(M)$  is a weight module. In many cases, the  $\sp_{2n}$-module $\mathfrak{W}(M)$  is $0$. For  example, if $M$ is a simple  weight module with a weight not in $-\frac{1}{2}\omega_n+\Z^n$, one can easily see that $\mathfrak{W}(M)=0$. We also note that $\mathfrak{W}(M)=M$ if
 $M$ is a simple  weight $\sp_{2n}$-module with a weight  in $-\frac{1}{2}\omega_n+\Z^n$.
If $M$ is a $U(\mh_n)$-torsion free module of finite rank when restricted to $U(\mh_n)$, then $\mathfrak{W}(M)$ is
a uniformly bounded weight module with $\text{supp}( \mathfrak{W}(M))= -\frac{1}{2}\omega_n+\Z^n$.

\subsection{Nilsson's modules}

Since $\mh_n$ is commutative,  $U(\mh_n)= \C[h_1,\cdots,h_n]$ as an associative algebra. In \cite{N2}, Nilsson constructed an $\sp_{2n}$-module structure on $U(\mh_n)$ as follows:
$$\aligned  h_i\cdot g&=h_ig,\\
X_{2\e_i}\cdot g&= (h_i-\frac{1}{2})(h_i-\frac{3}{2})\sigma_i^2(g),\\
X_{-2\e_i}\cdot g&=-\sigma_i^{-2}(g),\\
X_{\e_i+\e_j}\cdot g&=(h_i-\frac{1}{2})(h_j-\frac{1}{2})\sigma_i\sigma_j(g), i\neq j,\\
X_{-\e_i-\e_j}\cdot g&=-\sigma_i^{-1}\sigma_j^{-1}(g), i\neq j,\\
X_{\e_i-\e_j}\cdot g&=(h_i-\frac{1}{2})\sigma_i\sigma_j^{-1}(g), i\neq j,
\endaligned $$
where $g\in U(\mh_n) $ and $\sigma_i\in \text{Aut}(U(\mh_n))$ such that $\sigma_i(h_k)=h_k-\delta_{ik}$.
We denote by $N_{\mathbf{i}}$  this $\sp_{2n}$-module. It is easy to see that $N_{\mathbf{i}}$ is a Whittaker module with respect to the pair $(\sp_{2n},\mm_n)$ of type $\mathbf{i}=(\mathrm{i},\dots,\mathrm{i})$, where $\mathrm{i}$ is the imaginary number unit.
 In \cite{N2}, Nilsson showed  that any $\sp_{2n}$-module structure on  $U(\mh_n)$ is equivalent to $N_{\mathbf{i}}$ up to some automorphism of $\sp_{2n}$.
 We will show that any simple module in the block $\mathcal{WH}_{-\mathbf{i}}^{-\frac{1}{2}\omega_n}$  is equivalent to $N_{\mathbf{i}}$.

\section{Non-singular Whittaker modules}

In this section, we will characterize the  category $\mathcal{WH}_\mathbf{a}^{\mu}$ when $\mathbf{a}\in (\C^*)^n$ using the weighting functor. In this case, a module in $\mathcal{WH}_\mathbf{a}^{\mu}$ is  non-singular. For  convenience,  set $t^{\mathbf{m}}=t_1^{m_1}\cdots t_n^{m_n}$,
 $ h^{\mathbf{m}}=h_1^{m_1}\cdots h_n^{m_n}$,  for any $\mathbf{m}=(m_1,...,m_n)\in \Z^n$.

\subsection{The category $\mathcal{WH}_\mathbf{a}^{\mu}$}

We define  the total order on $\Z_{\geq 0}^n$ satisfying the condition: $\br<\bm$ if $|\br|<|\bm|$ or
$|\br|=|\bm|$ and
there is an $l\in \{1,\cdots, n\}$ such that $r_i=m_i$ when $1\leq i <l$ and $r_l< m_l$, where $|\bm|=\sum_{i=1}^n m_i$.  For each $\bm$, the set $\{\br \in \Z_{\geq 0}^n\mid \br< \bm\}$ is  finite. Hence as an ordered set $\Z_{\geq 0}^n$ is isomorphic to $\Z_{\geq 0}$.
For a nonzero $\bm \in \Z_{\geq 0}^n$, denote by $\bm'$ the predecessor
of $\bm$, i.e., $\bm'$ is the maximal element in $ \Z_{\geq 0}^n$ such that  $\bm'< \bm$.

The following lemma gives a rough characterization of modules in $\mathcal{WH}_\mathbf{a}^{\mu}$ which is important for the later discussions.

\begin{lemma}\label{v-f}
Any module $M$ in
$\mathcal{WH}_{\mathbf{a}}$ is a free  $U(\mh_n)$-module with the basis $\mathrm{wh}_{\mathbf{a}}(M)$.\end{lemma}
\begin{proof}
First, we show that $M=U(\mh_n)\mathrm{wh}_\ba(M)$.

 Denote   $Y^\mathbf{m}=(X_{-2\e_1}-a_1^2)^{m_1}\cdots (X_{-2\e_n}-a_n^2)^{m_n}$, for any $\bm\in \Z_+^n$.
 Then  the set $\{Y^{\bs}\mid \bs\in\Z_{\geq 0}^n\}$ forms a basis of $U(\mm_n)$.
 Using the hypothesis  that $a_i\neq 0$ for any $i$ and induction on $\bm$,  from $[h_i, X_{-2\e_i}]=-2X_{-2\e_i}$, we can show that for any $\bm,\bs\in \Z_{\geq 0}^n$  and nonzero $v\in \text{wh}_{\ba}(M)$, we have that $Y^{\bs} h^{\bm}v=0$ whenever $\bs>\bm$,
$Y^{\bm} h^{\bm} v= k_{\bm} v$ for some nonzero scalar $k_{\bm}$.

For each $\bm \in \Z_{\geq 0}^n$,  let  $\mathbb{I}_{\bm}$ be the ideal of $U(\mm_n)$ spanned by  $Y^{\bs}$ with $\bs>\bm$, and $M_{\bm}=\{w\in M\mid \mathbb{I}_{\bm}w=0\}$. Clearly $\mathrm{wh}_{\ba}(M)=M_{\bo}$.
For any nonzero $w\in M$, by the definition of $M$, there is an
$\bm \in \Z_{\geq 0}^n$ such that $w\in M_{\bm}\setminus M_{\bm'}$, i.e.,
$Y^{\bm} w\neq 0$ and $Y^{\bs} w=0$ for any $\bs>\bm$. So $Y^{\bm} w\in \mathrm{wh}_{\ba}(M)$. We call $\bm$ the degree of $w$. By the above discussion, $Y^{\bm} h^{\bm} Y^{\bm} w=k_{\bm}Y^{\bm} w  $.
We use  induction on the degree  $\bm$ of $w$ to show that $w\in U(\mh_n)\mathrm{wh}_{\ba}(M)$. Let $w'=w-\frac{1}{k_{\bm}}h^{\bm}Y^{\bm}w$. Then
$$Y^{\bm} w'=Y^{\bm} w-\frac{1}{k_{\bm}}Y^{\bm} h^{\bm}Y^{\bm}w=0. $$
This implies that the degree of $w'$ is smaller than $\bm$. By the induction
hypothesis, $w' \in U(\mh_n)\mathrm{wh}_{\ba}(M)$. Consequently $w\in U(\mh_n)\mathrm{wh}_{\ba}(M)$.

Next we show  that $\mathrm{wh}_{\mathbf{a}}(M)$ is a basis of $M$
as a free $U(\mh_n)$-module.  Suppose that $\{v_i | i=1,\dots,k\}$ is a basis of the vector
space $\mathrm{wh}_{\mathbf{a}}(M)$. We need to show that $\{ h^{\mathbf{m}}v_i\mid \mathbf{m}\in \Z_{+}^n, i=1,\dots,k\}$ is linearly independent.
Suppose that $w:=\sum_{\mathbf{r} \leq \mathbf{m}} \sum_{i=1}^k c_{\mathbf{r},i}h^{\mathbf{r}} v_i=0$. Then from
$Y^{\m} w=0$, we see that $c_{\m,i}=0$. Consequently by induction on $\m$,
$c_{\r,i}=0$ for any $\r < \m$ and $i$. Thus  $\{ h^{\mathbf{m}}v_i\mid \mathbf{m}\in \Z_{+}^n, i=1,\dots,k\}$ is linearly independent. The proof is complete.

\end{proof}

With the characterizations of modules in Lemma \ref{v-f}, we can use the weighting functor and the category $\mathcal{B}$ of uniformly bounded weight modules to study $\mathcal{WH}_{\mathbf{a}}^{\mu}$.

\begin{lemma}We have the following statements.
\label{bounded}\begin{enumerate}[$($a$)$]
\item For any $\mathbf{a}\in (\C^*)^n$ and $\mu\not\in \Lambda^+$, if the block $\mathcal{WH}_\mathbf{a}^{\mu}$ is non-empty, then $\mu(h_i-h_{i+1})\in \Z_{\geq 0}$, for any $i\neq n$, $\mu(h_n)\in \frac{1}{2}+\Z$ and $\mu(h_{n-1}+h_{n})\in \Z_{\geq -2}$.
    \item
For any $\mu\in \mh_n^*$ and $\mathbf{a}\in (\C^*)^n$, if  $\mathcal{WH}_\mathbf{a}^{\mu}$ is non-empty, then $\mathcal{WH}_\mathbf{a}^{\mu}$ is equivalent to $\mathcal{WH}_{\mathbf{a}}^{-\frac{1}{2}\omega_n}$.
\end{enumerate}
\end{lemma}
\begin{proof}(a) Let $M\in\mathcal{WH}_{\mathbf{a}}^{\mu}$. By Lemma \ref{v-f}, $M$ is a free  $U(\mh_n)$-module of finite rank. Then
the module $\mathfrak{W}(M)$ is a bounded weight $\sp_{2n}$-module, i.e, $\mathfrak{W}(M)\in \mathcal{B}^{\mu}$. By Lemmas 9.1 and 9.2 in \cite{M}, one can prove (a). We should note that  the symbol $h_i$ in \cite{M} represents the simple coroots which are different from our $h_i$.

(b) For a $\lambda\in \Lambda^+$,
let $L(\lambda)$ be the simple $\sp_{2n}$-module  of highest weight $\lambda$. The condition
$\lambda\in \Lambda^+$ implies that $L(\lambda)$ is finite dimensional. Recall  that the translation
functor $T_{-\frac{1}{2}\omega_n}^{\mu}$  is defined by  $$T_{-\frac{1}{2}\omega_n}^{\mu}(M)=\{ v\in L(\lambda)\otimes M\mid (z-\chi_{\mu}(z))^k v=0, \text{for some}~~ k\in \Z_+, \forall \ z\in Z(\sp_{2n})\}, $$ for any $M\in \mathcal{WH}_{\mathbf{a}}^{-\frac{1}{2}\omega_n}$,
 If $\mathcal{WH}_{\mathbf{a}}^{\mu}$ is nonempty, then by the proof of Lemma 9.2 in \cite{M},  we can choose $\lambda\in \Lambda^+$ such that
the functor $T_{-\frac{1}{2}\omega_n}^{\mu}$  gives an equivalence between $\mathcal{WH}_{\mathbf{a}}^{-\frac{1}{2}\omega_n}$ and $\mathcal{WH}_{\mathbf{a}}^{\mu}$, see also \cite{BG}.
\end{proof}

In order to study the category $\mathcal{WH}_{\mathbf{a}}^{-\frac{1}{2}\omega_n}$, we use the Weyl algebra $\cd_n$.
Let $A_{n}= \C [t_1, \dots, t_{n}]$ be the polynomial algebra in $n$ variables.
Then the subalgebra of $\text{End}_{\C}(A_{n})$  generated by $$\{t_i,\p_i:=\frac{\partial}{\partial t_i}\mid 1\leq i\leq n\}  $$ is called the Weyl algebra $\cd_{n}$ over $A_{n}$.  Namely,
$\cd_{n}$  is the unital  associative algebra
over $\mathbb{C}$ generated by $t_1,\dots,t_n$,
$\partial_1,\dots,\partial_n$ subject to
the following relations
$$[\partial_i, \partial_j]=[t_i,t_j]=0,\qquad [\partial_i,t_j]=\delta_{i,j},\ 1\leq i,j\leq n.$$

 Let $\cd_n^{ev}$ be the subalgebra of $\cd_n$ spanned by $$\{t^\alpha\p^\beta\mid \alpha,\beta\in \Z_+^n, |\alpha|+|\beta|\in 2\Z_+ \},$$ where $\p^\beta=\p_1^{\beta_1}\cdots \p_n^{\beta_n}$. We call $\cd_n^{ev}$ the even Weyl algebra of rank $n$.
In the following lemma, we recall a differential operator realization of $\sp_{2n}$, see \cite{BL}.
\begin{lemma}\label{hom}
 The map \begin{equation}\label{weil-rep}
\aligned \theta_0: \,\, U(\sp_{2n})&\rightarrow  \cd_n^{ev},\\
X_{\e_i+\e_j}&\mapsto  t_it_j,\\
 X_{\e_i-\e_j}& \mapsto  t_i\partial_j, \ \  i\neq j,\\
 h_i&\mapsto t_i\partial_i+\frac{1}{2},\\
X_{-\e_i-\e_j}& \mapsto -\p_i\p_j,1\leq i, j \leq n,
          \endaligned
          \end{equation}
  defines a surjective algebra homomorphism.
  \end{lemma}

Let $P_n$ be the unital subalgebra of $A_n$ generated by $t_it_j$, $i,j\in \{1,\cdots,n\}$. By Lemma \ref{hom},  $P_n$ can be made to be an $\sp_{2n}$-module called the Weil representation, see \cite{M}. It is easy to see that $P_n$  is equivalent to the simple highest weight module $L(-\frac{1}{2}\omega_n)$  of the highest  weight $-\frac{1}{2}\omega_n$ up to an involution of $\sp_{2n}$, where $\omega_n=\sum_{i=1}^n\e_i$ is the $n$-th fundamental weight of $\sp_{2n}$.

By Theorem 5.2 in \cite{GS1}, we obtain the following description of $\mathcal{B}^{-\frac{1}{2}\omega_n}$.

\begin{lemma}\label{ker-lem} If $M$ is a module in $\mathcal{B}^{-\frac{1}{2}\omega_n}$, then $\ker \theta_0 M=0$.
\end{lemma}

Using Lemma \ref{ker-lem} and the weighting functor, we show that any module in
$\mathcal{WH}_\mathbf{a}^{-\frac{1}{2}\omega_n}$ is actually a $\cd^{ev}_n$-module.

\begin{lemma}\label{ker} If $M$ is a module in $\mathcal{WH}_\mathbf{a}^{-\frac{1}{2}\omega_n}$,  then $\ker \theta_0 (M)=0$, i.e. $M$ is
a $\cd_n^{ev}$-module.
\end{lemma}
\begin{proof}By lemma \ref{v-f}, $M$ is a free $U(\mh_n)$-module.
The module $\mathfrak{W}(M)$ is a uniformly bounded weight $\sp_{2n}$-module, i.e. $\mathfrak{W}(M)\in \mathcal{B}^{-\frac{1}{2}\omega_n}$. By
Lemma \ref{ker-lem},  $\ker \theta_0 (\mathfrak{W}(M))=0$. So $\ker \theta_0 M\subset I_{\alpha -\frac{1}{2}\omega_n}M$ for any $\alpha\in \Z^n$. Since $M$ is a free $U(\mh_n)$-module of finite rank, we have that
$\cap_{\alpha\in \Z^n}(I_{\alpha-\frac{1}{2}\omega_n} M)=0$. So $\ker \theta_0 M=0$.

\end{proof}

 Let $\mathcal{W}^{\mathbf{a}}$ be the category of
$\cd_n^{ev}$-modules $V$ such that $\p_i^2-a^2_i$ acts locally nilpotently on $V$ for any $i\in\{1,\dots,n\}$, and $\text{wh}'_{\ba}(V):=\{ v\in V \mid \p_i^2v=a^2_iv,\  \forall\ i=1,\dots,n\}$ is finite dimensional.
 Then by Lemma \ref{ker}, we have the following equivalence.

\begin{theorem}\label{eth}The category  $\mathcal{WH}_{\mathbf{a}}^{-\frac{1}{2}\omega_n}$ is equivalent to
the category $\mathcal{W}^{\mathbf{a}}$ of $\cd^{ev}_n$-modules.
\end{theorem}

\subsection{Modules over the even Weyl algebra}
By Theorem \ref{eth}, we need to study the category $\mathcal{W}^{\mathbf{a}}$  for $\cd^{ev}_n$.
For a
$\mathbf{b}=(b_1,\dots, b_n)\in (\C^*)^n$ such that $b_i^2=a^2_i$ for all $i$, we define a $\cd_n^{ev}$-module $M_{\bb}:=\C[x_1,\cdots,x_n]$ as follows:
$$\aligned
\p_i\p_j x^{\bm}&=b_ib_j \tau_i^{-1}\tau_j^{-1}( x^{\bm}),\\
t_it_j x^{\bm}&=b_i^{-1}b_j^{-1} x_ix_j\tau_i\tau_j( x^{\bm}) ,\ \ i\neq j,\\
t_i^2 x^{\bm}&=b_i^{-2}x_i \tau_i^2( x^{\bm})-x_ix^{\bm},\\
t_i\p_j x^{\bm}&= b_i^{-1}b_jx_i\tau_i\tau_j^{-1}( x^{\bm}),
\endaligned$$
where $x^{\bm}=x_1^{m_1}\cdots x_n^{m_n}, \tau_i\in \text{Aut}(\C[x_1,\cdots,x_n])$ such that $\tau_i(x_k)=x_k-\delta_{ik}$.

\begin{lemma}\label{ss}
\begin{enumerate}[$($a$)$]
\item Any simple module $M$ in $\mathcal{W}^{\mathbf{a}}$ is isomorphic to $M_{\mathbf{b}}$, where $\mathbf{b}\in (\C^*)^n$ such that $b_i^2=a^2_i$ for all $i$.
\item The category $\mathcal{W}^{\mathbf{a}}$ of $\cd_n^{ev}$-modules is semi-simple.
\end{enumerate}
\end{lemma}

\begin{proof}
(a) Suppose that $N$ is a nonzero submodule of $M_{\mathbf{b}}$.
Since $\p_i^2-a_i^2$ decreases the degrees of $x_i$, we must have that $x^0:=1\in N$. Note that $x^0$ generates $M_{\bb}$. So $N=M_{\bb}$, $M_{\bb}$ is simple.

Suppose that $M$ is a simple module in  $\mathcal{W}^{\mathbf{a}}$. Since $\text{wh}'_{\ba}(M)$ is finite dimensional and $[\p_i,\p_j]=0$, there are  a
nonzero $v\in M$ and $\mathbf{b}\in (\C^*)^n$ such that $\p_i\p_jv=b_ib_jv$ and $b_i^2=a_i^2$ for all $i,j$.
We can define a $\cd_n^{ev}$-module isomorphism $\tau$ from $M$ to $M_{\bb}$
such that $\tau((t_1\p_1)^{m_1}\cdots (t_n\p_n)^{m_n}v)=x^{\bm}$, for all $\bm\in \Z_+^n$.  So $M\cong M_{\bb}$.

(b) It suffices to show that $\text{Ext}^1_{\cd_n}(M_\mathbf{b}, M_\mathbf{b'})=0$. If there are $i,j$ such that $b_ib_j\neq b'_ib'_j$, then
from that the eigenvalues of $\p_i\p_j$ on $M_{\bb}$ and $M_{\bb'}$ are different, $\text{Ext}^1_{\cd_n}(M_\mathbf{b}, M_\mathbf{b'})=0$. So it's suffices to consider that $\bb=\bb'$.
We will show that  the short exact sequence
\begin{equation}\label{seq}
0 \rightarrow M_{\bb}\xrightarrow{\alpha} V \xrightarrow{\beta}  M_{\bb} \rightarrow 0
\end{equation}
of $\cd_n^{ev}$-modules is split. By the similar proof in Lemma \ref{v-f}, we can show that $V=\C[t_1\p_1,\cdots,t_n\p_n]\otimes \text{wh}'_{\bb}(V)$, and $\dim \text{wh}'_{\bb}(V)=2$, since
$\dim \text{wh}'_{\bb}(M_{\bb})=1$. Explicitly we can replace $h_i$ and $X_{-2\epsilon_i}$ by $t_i\p_i$ and $\p_i^2$ respectively in the proof of Lemma \ref{v-f}.
Choose $v\in \text{wh}'_{\bb}(V)\setminus \alpha(M_{\bb})$. By the proof of (a), the submodule $\cd_n^{ev}v\cong M_{\bb}$. Note that $M_{\bb}$ is a free $\C[t_1\p_1,\cdots,t_n\p_n]$-module of rank one.   So the sequence (\ref{seq}) is split.
 The proof is complete.
\end{proof}

Combining Theorem \ref{eth} and Lemma \ref{ss}, we obtain the following characterization of $\mathcal{WH}_{\mathbf{a}}^{-\frac{1}{2}\omega_n}$.
\begin{theorem}\label{semi-simple}The category  $\mathcal{WH}_{\mathbf{a}}^{-\frac{1}{2}\omega_n}$ is semi-simple.
Moreover, any simple module in $\mathcal{WH}_{\mathbf{a}}^{-\frac{1}{2}\omega_n}$  is isomorphic to $M_{\mathbf{b}}$, where $\mathbf{b}=(b_1,\dots, b_n)\in (\C^*)^n$ such that $b_i^2=a_i^2$ for all $i$, and $M_{\mathbf{b}}$ is  an $\sp_{2n}$-module under the map (\ref{weil-rep}).
\end{theorem}

The module $M_{\mathbf{b}}$
is equivalent to $M_{\mathbf{a}}$ up to an isomorphism of $\sp_{2n}$, for any $\mathbf{b}\in (\C^*)^n$ such that $b_i^2=a_i^2$ for all $i$. Moreover the $\sp_{2n}$-module $M_{\mathbf{i}}$ is isomorphic to the Nilsson's module $N_{\mathbf{i}}$. Then by Theorem \ref{semi-simple}, we have the following result.
\begin{corollary}Any simple module in $\mathcal{WH}_{\mathbf{i}}^{-\frac{1}{2}\omega_n}$  is equivalent to $N_{\mathbf{i}}$, up to an isomorphism of $\sp_{2n}$.
\end{corollary}

\section{General Weil representations}

In Section 3, we see that the algebra homomorphism $\theta_0$ from $U(\sp_{2n})$ to the Weyl algebra $\cd_n$ is useful for the study of representations of $\sp_{2n}$. In this section,   we will find more algebra homomorphisms from $U(\sp_{2n})$ to the Weyl algebra $\cd_n$.

\subsection{General Weil representations}

In the following lemma, we give a differential operators realization of $\sp_{2n}$ from any $f\in A_{n}$.
\begin{lemma}\label{hom1}
For any $f\in A_n$, the map \begin{equation}\label{weil-rep1}
\aligned \theta_f: \,\, U(\sp_{2n})&\rightarrow  \cd_n,\\
X_{\e_i+\e_j}&\mapsto  t_it_j,\\
 X_{\e_i-\e_j}& \mapsto  t_i\partial_j(f)+t_i\partial_j, \ \  i\neq j,\\
 h_i&\mapsto t_i\partial_i(f)+t_i\partial_i+\frac{1}{2},\\
X_{-\e_i-\e_j}& \mapsto -(\p_i(f)+\p_i)(\p_j(f)+\p_j),1\leq i, j \leq n,
          \endaligned
          \end{equation}
  defines an algebra homomorphism.
  \end{lemma}
  \begin{proof}
By Lemma \ref{hom},  $\theta_0$ is an algebra homomorphism.
For  general $f$,  the map $\theta_f $ is the composition $ \sigma_f \circ \theta_0$, where $\sigma_f$ is the algebra isomorphism of $\cd_n$ defined by $$t_i\mapsto t_i, \p_i \mapsto \p_i+\p_i(f).$$
Therefore $\theta_f$ is an  algebra homomorphism.

  \end{proof}

 Next  we will  show that $\theta_f$ in Lemma
\ref{hom1} exhausts all algebra  homomorphisms $\theta$ from $U(\sp_{2n})$ to $\cd_n$ such that $\theta(X_{\e_i+\e_j})=t_it_j$ for any $i,j$.

Firstly we give some formulas in $U(\sp_{2n})$ that will be used in the subsequent text.

\begin{lemma}For any $k\in \N$, these formulas hold as follows:
\begin{enumerate}
\item $[X_{\e_i-\e_j},X_{2\e_l}^k]=\delta_{jl}2kX_{2\e_l}^{k-1}X_{\e_i+\e_j}, 1\leq i \neq j \leq n$; \\
\item $[X_{-\e_i-\e_j},X_{2\e_l}^k]=-2kX_{2\e_l}^{k-1}\big(\delta_{jl}X_{\e_j-\e_i}+\delta_{il}X_{\e_i-\e_j}\big), 1\leq i\neq j\leq n;$\\
\item $[X_{-2\e_i},X_{2\e_l}^k]=-\delta_{il}4kX_{2\e_l}^{k-1}\big(h_i+k-1\big), 1\leq i \leq n.$
\end{enumerate}
\end{lemma}
\begin{proof}(1) According to $[X_{\e_i-\e_j},X_{2\e_l}]=2\delta_{jl}X_{\e_i+\e_j}$ and $[X_{\e_i+\e_l},X_{2\e_l}]=0$, we can compute that
\begin{align*}[X_{\e_i-\e_j},X_{2\e_l}^k]
=&\sum_{t=1}^kX_{2\e_l}^{k-t}[X_{\e_i-\e_j},X_{2\e_l}]X_{2\e_l}^{t-1}\\
=&\sum_{t=1}^k2\delta_{jl}X_{2\e_l}^{k-1}X_{\e_i+\e_j}
=\delta_{jl}2kX_{2\e_l}^{k-1}X_{\e_i+\e_j}.
  \end{align*}

(2) From $[X_{-\e_i-\e_j},X_{2\e_l}]=-2\big(\delta_{jl}X_{\e_j-\e_i}+\delta_{il}X_{\e_i-\e_j}\big)$ and $$ [X_{\e_i-\e_j},X_{2\e_l}^{t-1}]=\delta_{jl}2(t-1)X_{2\e_l}^{t-2}X_{\e_i+\e_j},$$ we have that  $$\begin{aligned}&[X_{-\e_i-\e_j},X_{2\e_l}^k]\\
=&\sum_{t=1}^kX_{2\e_l}^{k-t}[X_{-\e_i-\e_j},X_{2\e_l}]X_{2\e_l}^{t-1}\\
=&\sum_{t=1}^k-2X_{2\e_l}^{k-t}\big(\delta_{jl}X_{\e_j-\e_i}X_{2\e_l}^{t-1}+\delta_{il}X_{\e_i-\e_j}X_{2\e_l}^{t-1}\big)\\
=&\sum_{t=1}^k-2X_{2\e_l}^{k-t}\big(\delta_{il}X_{2\e_l}^{t-1}X_{\e_i-\e_j}+\delta_{jl}X_{2\e_l}^{t-1}X_{\e_j-\e_i}\\
  &\quad +2(t-1)X_{2\e_l}^{t-2}\delta_{il}\delta_{jl}(X_{\e_i+\e_l}+X_{\e_j+\e_l})\big)\\
=&-2kX_{2\e_l}^{k-1}\big(\delta_{il}X_{\e_i-\e_j}+\delta_{jl}X_{\e_j-\e_i}\big)-2k(k-1)\delta_{il}\delta_{jl}X_{2\e_l}^{k-2}(X_{\e_i+\e_l}+X_{\e_j+\e_l}).
  \end{aligned}$$
Thus, by $i\neq j$, we can see  that
$$[X_{-\e_i-\e_j},X_{2\e_l}^k]=-2kX_{2\e_l}^{k-1}\big(\delta_{jl}X_{\e_j-\e_i}+\delta_{il}X_{\e_i-\e_j}\big),\quad i\neq j.$$

 (3) By the similar computation in (2), we can obtain that
$$\aligned \ [X_{-2\e_i},X_{2\e_l}^k]&=-2kX_{2\e_l}^{k-1}2\delta_{il}h_i-2\delta_{il}k(k-1)X_{2\e_l}^{k-2}2X_{2\e_l}
\\&=-\delta_{il}4kX_{2\e_l}^{k-1}\big(h_i+k-1\big).\endaligned$$
The proof is complete.

\end{proof}
In the following lemma, we give a preliminary description of algebra homomorphisms from $U(\sp_{2n})$ to $\cd_n$.
\begin{lemma}\label{1} If $\theta$ is an algebra homomorphism from $U(\sp_{2n})$ to $\cd_n$ such that $$\theta(X_{\e_i+\e_j})=t_it_j,$$ for any $i,j$, then there exist $p_{ij},q_{ij}\in A_{n}$ such that
\begin{enumerate}
\item $\theta(X_{\e_i-\e_j})= p_{ij}+t_i\partial_j, 1\leq i\neq j\leq n$;
\item $\theta(h_i)= p_{ii}+t_i\partial_i, 1\leq i\leq n$;
\item $\theta(X_{-2\e_i}) =q_{ii}+(1-2 p_{ii})t_i^{-1}\partial_i-\partial_i^2$, $1\leq i\leq n$;
\item $\theta(X_{-\e_i-\e_j})= q_{ij}-p_{ij}t_i^{-1}\partial_i-p_{ji}t_j^{-1}\partial_j-\partial_i\partial_j, 1\leq i \neq j\leq n$.
\end{enumerate}
\end{lemma}
\begin{proof}We consider the action of $\theta(X_{\alpha})$ on $R_n:=\C[t_1^{\pm}, \dots, t_n^{\pm 1}]$.
For the convenience, we denote $X_{2\e_1}^{m_1}\cdots X_{2\e_n}^{m_n}$ by $X^\mathbf{m}$, and
$\theta(X_\alpha)(g(t))$  by $X_\alpha\cdot g(t)$ for any $\alpha\in\Delta$ and $g(t)\in R_n$.

 By definition, $X_{\e_i+\e_j}\cdot t^{\m}=t_it_jt^{\m}$, for any $\bm\in \Z^n$.

(1) Define $p_{ij}=X_{\e_i-\e_j}\cdot 1$.  From $[X_{\e_i-\e_j},X_{2\e_l}^k]=\delta_{jl}2kX_{2\e_l}^{k-1}X_{\e_i+\e_j}$, we obtain that
\begin{align*}
X_{\e_i-\e_j}\cdot t^{2\m}&=X_{\e_i-\e_j}\cdot X^\mathbf{m}\cdot 1\\
&=[X_{\e_i-\e_j},X^\mathbf{m}]\cdot 1+ X^\mathbf{m}\cdot X_{\e_i-\e_j}\cdot 1\\
&=2m_jX^{\mathbf{m}-e_j}X_{\e_i+\e_j}\cdot 1+X^\mathbf{m}\cdot X_{\e_i-\e_j}\cdot 1\\
&=2m_jt^{2\m+e_i-e_j}+p_{ij}t^{2\m}\\
&=(p_{ij}+t_i\partial_j)(t^{2\m}).
\end{align*}

(2) It is similar as (1)

(3) For $i=j$, by $[X_{-2\e_i},X_{2\e_l}^k]=-\delta_{il}4kX_{2\e_l}^{k-1}\big(h_i+k-1\big)$, we can calculate that
\begin{align*}
[X_{-2\e_i},X^\mathbf{m}]&=[X_{-2\e_i},\prod_{l=1}^nX_{2\e_l}^{m_l}]\\
&=\prod_{j=1}^{i-1}X_{2\e_j}^{m_j}[X_{-2\e_i},X_{2\e_i}^{m_i}]\prod_{s=i+1}^{n}X_{2\e_s}^{m_s}\\
&=\prod_{j=1}^{i-1}X_{2\e_j}^{m_j}\Big(-4m_iX_{2\e_i}^{m_i-1}\big(h_i+m_i-1\big)\Big)\prod_{s=i+1}^{n}X_{2\e_s}^{m_s}\\
&=-4m_iX^{\mathbf{m}-e_i}(h_i+m_i-1).
\end{align*}

Hence, set $q_{ii}=X_{-2\e_i}\cdot 1 $, we can determine the action of $X_{-2\e_i}$ as follows:
\begin{align*}
X_{-2\e_i}\cdot t^{2\m}&=X_{-2\e_i}\cdot X^\mathbf{m}\cdot 1 \\
&=[X_{-2\e_i},X^\mathbf{m}]\cdot 1+X^\mathbf{m}\cdot X_{-2\e_i} \cdot 1\\
&=-4m_iX^{\mathbf{m}-e_i}(h_i+m_i-1)\cdot 1+X^\mathbf{m}\cdot X_{-2\e_i} \cdot 1\\
&=(1-2 p_{ii})t_i^{-1}\partial_it^{2\m}-\partial_i^2t^{2\m}+q_{ii}t^{2\m}\\
&=\big(q_{ii}+(1-2 p_{ii})t_i^{-1}\partial_i-\partial_i^2\big)t^{2\m}.
\end{align*}

(4) When  $i\neq j$, let's suppose $i< j$, and the case for $i> j$ is similar. According to  $[X_{-\e_i-\e_j},X_{2\e_l}^k]=
-2kX_{2\e_l}^{k-1}\big(\delta_{jl}X_{\e_j-\e_i}
+\delta_{il}X_{\e_i-\e_j}\big)$,
we obtain  that
\begin{align*}
&[X_{-\e_i-\e_j},X^\mathbf{m}]=[X_{-\e_i-\e_j},\prod_{l=1}^nX_{2\e_l}^{m_l}]\\
=&-2m_iX^{\mathbf{m}-e_i}X_{\e_i-\e_j}-4m_im_jX^{\mathbf{m}-e_i-e_j}X_{\e_i+\e_j}
  -2m_jX^{\mathbf{m}-e_j}X_{\e_j-\e_i}.
\end{align*}

 Denote $q_{ij}=X_{-\e_i-\e_j} \cdot 1$. Then, we have
\begin{align*}
X_{-\e_i-\e_j}\cdot t^{2\m}&=X_{-\e_i-\e_j}\cdot X^\mathbf{m}\cdot 1 \\
&=[X_{-\e_i-\e_j},X^\mathbf{m}]\cdot 1+X^\mathbf{m}\cdot X_{-\e_i-\e_j} \cdot 1\\
&=\big(q_{ij}-p_{ij}t_i^{-1}\partial_i-p_{ji}t_j^{-1}\partial_j-\partial_i\partial_j\big) (t^{2\m}).
\end{align*}
Consequently the proof is completed.

\end{proof}

The following easy lemma will be used in the proof of Theorem \ref{main-th}
\begin{lemma}\label{2} For any $i\in \{1,\cdots,n\}$ and $p\in A_n$, the differential equation
$$(t_i\p_i+1)(q)=p$$
has at most one solution $q$ in  $A_n$.\end{lemma}
\begin{proof} The proof follows from the fact  that the map $t_i\p_i+1: A_n\rightarrow A_n$ is injective.

\end{proof}

 Combining the above preparatory arguments, we can give all possible algebra homomorphisms from $U(\sp_{2n})$ to the Weyl algebra $\cd_n$ which map each root vector $X_{\e_i+\e_j}$  to $t_it_j$.

\begin{theorem}\label{main-th}If $\theta$ is an algebra homomorphism from $U(\sp_{2n})$ to $\cd_n$ such that $$\theta(X_{\e_i+\e_j})=t_it_j,$$ for any $i,j$,  then there is an $f\in A_n$ such that $\theta=\theta_f$, that is
\begin{enumerate}
\item $\theta(X_{\e_i+\e_j})=  t_it_j$ ,
\item $\theta(X_{\e_i-\e_j})= t_i\partial_j(f)+t_i\partial_j,  i\neq j$,
\item $\theta(h_i)= t_i\partial_i(f)+t_i\partial_i+\frac{1}{2}$,
\item $\theta(X_{-\e_i-\e_j})=-(\p_i(f)+\p_i)(\p_j(f)+\p_j)$,
\end{enumerate} where $1\leq i,j \leq n$.
\end{theorem}

\begin{proof} It is sufficient to determine $p_{ij}, q_{ij}$ in Lemma $\ref{1}$. For any $f\in A_{n}$, we set $f_{(i)}=\partial_i(f)$.

Specifically, according to $[h_i,h_j]=0$, which means $[p_{ii}+t_i\partial_i,p_{jj}+t_j\partial_j]=0$, we can calculate $$t_i\partial_i(p_{jj})=t_j\partial_j(p_{ii}), 1\leq i,j \leq n.$$

Then we can obtain that  $p_{ii}=t_if_{(i)}+c_i$ for some $f\in A_n$ and $\c=(c_1,\ldots,c_n)\in \C^n$.

By $[X_{\e_i-\e_j},X_{\e_j-\e_i}]=h_i-h_j$, we deduce that
$$p_{ij}t_j\partial_i+t_i\partial_jp_{ji}-p_{ji}t_i\partial_j-t_j\partial_ip_{ij}=p_{ii}-p_{jj}.$$
Consequently we have $c_i=c_j$ for any $1\leq i,j \leq n$. Set $b=c_i$ for any $1\leq i\leq n$, then $p_{ii}=t_if_{(i)}+b$.

For $i\neq j$, $[h_i,X_{\e_i-\e_j}]=X_{\e_i-\e_j}$ implies that $(t_i\partial_i-1)(p_{ij})=t_i\partial_j(p_{ii})$. Solving this differential equation, we can check that $p_{ij}=t_if_{(j)}$ is the unique solution of it.

For $i\neq j$, from $[h_i,X_{-\e_i-\e_j}]=-X_{-\e_i-\e_j}$, we get
\begin{equation}\label{eq} t_i\partial_i(q_{ij})+q_{ij}=
-p_{ij}t_i^{-1}\partial_i(p_{ii})-p_{ji}t_j^{-1}\partial_j(p_{ii})-\partial_i\partial_j(p_{ii}).\end{equation}

We can calculate it more explicitly,
\begin{align*}
&t_i\partial_i(q_{ij})+q_{ij}\\
=&-p_{ij}t_i^{-1}\partial_i(p_{ii})-p_{ji}t_j^{-1}\partial_j(p_{ii})-\partial_i\partial_j(p_{ii}) \\
=&-\Big(t_if_{(j)}t_i^{-1}\partial_i\big(t_if_{(i)}+b\big)
+t_jf_{(i)}t_j^{-1}\partial_j\big(t_if_{(i)}+b\big)+\partial_i\partial_j\big(t_if_{(i)}+b\big)\Big)\\
=&-\Big(f_{(i)}f_{(j)}+t_if_{(j)}\partial_i(f_{(i)})+f_{(i)}t_i\partial_j(f_{(i)})+\partial_j(f_{(i)})
+t_i\partial_i\partial_j(f_{(i)})\Big).
\end{align*}
We can check  that $q_{ij}=-f_{(i)}f_{(j)}-\partial_j(f_{(i)})$ is the unique solution of the equation (\ref{eq}).

Similarly, for $i=j$, we can get $q_{ii}=-f_{(i)}^2-\partial_i(f_{(i)})-(2b-1)t_i^{-1}f_{(i)}$.

Furthermore, for $i\neq j$, by $[X_{\e_i-\e_j},X_{-2\e_i}]=-2X_{-\e_i-\e_j}$, we deduce that
$$t_i\partial_j(q_{ii})-(1-2 p_{ii})t_i^{-1}\partial_i(p_{ij})+\p_i^2(p_{ij})=-2q_{ij}.$$
So  we have $(1-2b)t_i^{-1}f_{(j)}=0$, that is, $b=\frac{1}{2}$ and $q_{ij}=-f_{(i)}f_{(j)}-\partial_j(f_{(i)})$ for $1\leq i,j \leq n$. Therefore from Lemma
\ref{1}, we can complete the proof.


\end{proof}

\subsection{New $\sp_{2n}$-modules}

Recall that $P_n$ is the unital subalgebra of $A_n$ generated by $t_it_j$, for $i,j\in \{1,\cdots,n\}$. For any  $f\in P_n$,
via the homomorphism $\theta_f$ in (\ref{weil-rep}), $P_n$ becomes an $\sp_{2n}$-module
$P_n^f$.  In Theorem \ref{main-th}, we actually have classified all $\sp_{2n}$-module structures on $P_n$ satisfying $X_{\e_i+\e_j}\cdot g(t)=t_it_j g(t)$ for any $i,j\in \{1,\cdots,n\}, g(t)\in P_n$.

\begin{corollary} If there is an $\sp_{2n}$-module structure on $P_n$  such that
 $X_{\e_i+\e_j}\cdot g(t)=t_it_j g(t)$ for any $i,j\in \{1,\cdots,n\}, g(t)\in P_n$, then this
 $\sp_{2n}$-module structure
is isomorphic to $P_n^f$ for some $f\in P_n$.
\end{corollary}

%
%
%
%
%
%
%

The following proposition gives a description of $P_n^f$.

\begin{proposition}Let $f, f'\in P_n$.
\begin{enumerate}
\item The   $\sp_{2n}$-module $P_n^f$ is simple.
\item As $\sp_{2n}$-modules, $P_n^f\simeq P_n^{f'}$ if and only if $f-f'\in \C$.
\end{enumerate}
\end{proposition}
\begin{proof}
(1) One can check directly that $P_n$ is a
simple $\cd_n^{ev}$-module. Note that the image of $\theta_f$ is
$\cd_n^{ev}$. So $P_n^f$ is a simple  $\sp_{2n}$-module.

(2) The sufficiency is obvious, it is enough to consider the necessity. Let $\varphi: P_n^f \rightarrow P_n^{f'}$ be an $\sp_{2n}$-module isomorphism. Following $\varphi(X_{\e_i+\e_j}\cdot t^0)=X_{\e_i+\e_j}\cdot\varphi(t^0)$, we know that $\varphi(t^0)\neq 0$.
From $\theta_f(X_{\e_i+\e_j})=t_it_j$ and $\varphi(X_{\e_i+\e_j}\cdot t^0)=X_{\e_i+\e_j}\cdot\varphi( t^0)$, we see that $\varphi(g(t))=g(t)\varphi(t^0)$ for any $g(t)\in P_n$. Then for any $i\neq j$, we have that
$$\aligned
0&=\varphi(X_{\e_i-\e_j}\cdot t^0)-X_{\e_i-\e_j}\cdot\varphi(t^0)\\
&=\varphi(t_if_{(j)})-(t_if'_{(j)}+t_i\p_j)\cdot\varphi(t^0)\\
&=-t_i\big(\p_j+f'_{(j)}-f_{(j)}\big)\varphi(t^0).
\endaligned$$
Consequently we see that $\p_j+f'_{(j)}-f_{(j)}$ is not injective on $P_n$. So we have $f_{(j)}=f'_{(j)}$ and $ \varphi(t^0)\in \C^* t^0$ for any $j\in\{1,\dots, n\}$. That is $f-f'\in \C$, which completes the proof.

\end{proof}

Therefore we have constructed several simple modules $P_n^f$ over $\sp_{2n}$ generalizing the Weil representation.

\vspace{2mm}
\noindent

\vspace{0.2cm}

\noindent Y. Li: School of Mathematics and Statistics, Henan University, Kaifeng
475004, China. Email: 897981524@qq.com

\vspace{0.2cm}

\noindent J. Zhao: School of Mathematics and Statistics, Henan University, Kaifeng
475004, China. Email: zhaoj@henu.edu.cn

\vspace{0.2cm}

\noindent Y. Zhang: School of Mathematics and Statistics, Henan University, Kaifeng
475004, China. Email: zhangyy17@henu.edu.cn

\vspace{0.2cm}

 \noindent G. Liu: School of Mathematics and Statistics,
and  Institute of Contemporary Mathematics,
Henan University, Kaifeng 475004, China. Email: liugenqiang@henu.edu.cn

\end{document}